\documentclass[a4paper,twoside,11pt]{article}
\usepackage{a4wide, fancyhdr, amsmath, amssymb, mathtools, yfonts, mathrsfs}
\usepackage{graphicx}
\usepackage{cases}
\usepackage{tikz}
\usepackage[all]{xy}
\usepackage[utf8]{inputenc}
\usepackage{amsthm}
\usepackage[english]{babel}
\usepackage{chngcntr}
\usepackage{comment}
\usepackage{ifthen}
\usepackage{calc}
\usepackage{hyperref}
\usepackage[OT2,T1]{fontenc}
\usepackage{authblk}
\usepackage{multirow}
\usepackage{rotating}
\numberwithin{equation}{section}
\DeclareSymbolFont{cyrletters}{OT2}{wncyr}{m}{n}
\DeclareMathSymbol{\Sha}{\mathalpha}{cyrletters}{"58}

\providecommand{\keywords}[1]{{2010}\textit{ Mathematics Subject Classification. }#1}

\newcommand{\leg}[2]{\genfrac{(}{)}{}{}{#1}{#2}}
\newcommand{\fpr}[2]{\genfrac{[}{]}{}{}{#1}{#2}_4}

\newcommand{\ZZ}{\mathbb{Z}}
\newcommand{\QQ}{\mathbb{Q}}
\newcommand{\OO}{\mathcal{O}}
\DeclareMathOperator{\CL}{Cl}

\newcommand\FF{\mathbb{F}}

\newcommand\pp{\mathfrak{p}}

\newcommand\PP{\mathfrak{P}}

\newcommand\qq{\mathfrak{q}}

\newcommand\rrr{\mathfrak{r}}

\newcommand\pt{\mathfrak{t}}

\DeclareMathOperator{\rk}{rk}

\DeclareMathOperator{\Gal}{Gal}
\DeclareMathOperator{\Norm}{\mathfrak{N}}
\DeclareMathOperator{\Frob}{Frob}

\DeclareMathOperator{\ord}{ord}

\newcommand\K{\mathscr{K}}
\newcommand\sP{\mathscr{P}}
\newcommand\sE{\mathscr{E}}
\newcommand\sL{\mathscr{L}}
\newcommand\sM{\mathscr{M}}

\newtheorem{theorem}{Theorem}
\newtheorem{lemma}{Lemma}[section]
\newtheorem{prop}[lemma]{Proposition}

\newtheorem{conjecture}{Conjecture}

\DeclareMathOperator{\rank}{rank}

\title{Kuroda's formula and arithmetic statistics}
\author{Stephanie Chan\thanks{Gower Street, London, WC1E 6BT, United Kingdom, stephanie.chan.16@ucl.ac.uk}}
\author{Djordjo Milovic\thanks{Gower Street, London, WC1E 6BT, United Kingdom, djordjo.milovic@ucl.ac.uk}}
\affil{Department of Mathematics, University College London}

\date{\today}

\begin{document}

\maketitle

\begin{abstract} 
Kuroda's formula relates the class number of a multi-quadratic number field $K$ to the class numbers of its quadratic subfields $k_i$. A key component in this formula is the unit group index $Q(K) = [\OO_{K}^{\times}: \prod_i\OO_{k_i}^{\times}]$. We study how $Q(K)$ behaves on average in certain natural families of totally real biquadratic fields $K$ parametrized by prime numbers.
\end{abstract}

\section{Introduction}
The purpose of this paper is to study Kuroda's class number formula \cite{Kuroda, Kubota53, Kubota56} for real biquadratic fields from the standpoint of arithmetic statistics. The formula is as follows: let $K$ be a normal, totally real extension of $\QQ$ with $\Gal(K/\QQ)$ isomorphic to the Klein four-group $C_2\times C_2$, let $k_1$, $k_2$, and $k_3$ be the three quadratic subfields of $K$, let $h(K)$ (resp., $h(k_i)$, $i = 1, 2, 3$) denote the largest power of $2$ dividing the class number of $K$ (resp., of $k_i$, $i = 1, 2, 3$), and let $Q(K)$ denote the unit group index defined as
\[
Q(K) = [\OO_{K}^{\times}: \OO_{k_1}^{\times}\OO_{k_2}^{\times}\OO_{k_3}^{\times}].
\]
Then
\begin{equation}\label{eq:Kuroda}
h(K) = \frac{1}{4}Q(K)h(k_1)h(k_2)h(k_3).
\end{equation}
In general, the index $Q(K)$ can be $1$, $2$, or $4$ \cite{Kuroda}. A particular choice of $K$ that is natural from the standpoint of Gauss's genus theory and that appears in the literature \cite{Sime1, Sime2, AM, BLS, Yue1, Yue2, YZ} is 
\[
K = \QQ(\sqrt{p}, \sqrt{d}),
\]
where $p$ is a prime number and $d$ is a positive squarefree integer coprime to $p$. With this choice of $K$, we can now ask more precise statistical questions pertaining to the arithmetic objects appearing in \eqref{eq:Kuroda}. For instance, if we fix a positive squarefree integer $d$ and $i\in \{1, 2, 4\}$, then we may wish to determine the natural density, if it exists, of prime numbers $p$ such that $Q(K) = i$. In analogy with numerous works on $2$-parts of class groups (see for instance \cite{CohnLag, CohnLag2, Ste1}), we may further inquire if there exists a governing field $M_d/\QQ$, not depending on $p$, such that $Q(K)$ is determined by the Frobenius conjugacy class of $p$ in the Galois group $\Gal(M_d/\QQ)$. 

Before stating our results, we first establish our notation. Given integers $d_1, \ldots, d_k$, let $\K_{d_1, \ldots, d_k}$ denote the multiquadratic field $\QQ(\sqrt{d_1}, \ldots, \sqrt{d_k})$, and let $\CL(d_1, \ldots, d_k)$ (resp.\ $\CL^+(d_1, \ldots, d_k)$) denote the $2$-part of the class group (resp.\ the $2$-part of the narrow class group) of $\K_{d_1, \ldots, d_k}$. Similarly, let $H_{d_1, \ldots, d_k}$ (resp., $H_{d_1, \ldots, d_k}^+$) denote the $2$-Hilbert class field (resp., the narrow $2$-Hilbert class field) of $\K_{d_1, \ldots, d_k}$, i.e., the maximal unramified at all primes (resp., at all finite primes) abelian $2$-power-degree extension of $\K_{d_1, \ldots, d_k}$. Hence, by class field theory, the Artin map induces canonical isomorphisms 
\[
\CL(d_1, \ldots, d_k)\cong \Gal(H_{d_1, \ldots, d_k}/\K_{d_1, \ldots, d_k})\quad\text{and}\quad\CL^+(d_1, \ldots, d_k)\cong \Gal(H^+_{d_1, \ldots, d_k}/\K_{d_1, \ldots, d_k}).
\]
Let $h(d_1, \ldots, d_k) = |\CL(d_1, \ldots, d_k)|$ and $h^+(d_1, \ldots, d_k) = |\CL^+(d_1, \ldots, d_k)|$. For a finite abelian group $G$, a prime number $\ell$, and an integer $k\geq 1$, we define the $\ell^k$-rank of $G$ to be the non-negative integer
\[
\rk_{\ell^k}G = \dim_{\FF_{\ell}}(\ell^{k-1}G/\ell^kG).
\]
For an integer $n\geq 1$, let $C_n$ denote the cyclic group of order $n$. Let $\OO_L$ denote the ring of integers of a number field $L$. Finally, throughout this paper we will let $d$ denote a positive squarefree integer having exactly $t$ distinct prime factors, we will let $p$ denote a prime number coprime to $d$, and we will let $m_{d, p}$ denote the number of primes dividing $d$ that split completely in $\K_{p}/\QQ$.

We aim to study the natural density of the fibers of the map $\phi_{d}: p\mapsto Q(\K_{d, p})$. Since the Chebotarev Density Theorem is a ready-made tool for studying densities of prime numbers, it is of particular interest to determine when the map $\phi_d$ is \textit{Frobenian}, i.e., when there exists a normal extension $M_d/\QQ$, called a \textit{governing field}, and a class function $\varphi_d: \Gal(M_d/\QQ)\rightarrow \{1, 2, 4\}$ such that 
$$
\phi_d(p) = \varphi_d(\Frob_{M_d/\QQ}(p))
$$
for all primes $p$ that are unramified in $M_d/\QQ$ (here $\Frob_{M_d/\QQ}(p)$ denotes the Frobenius conjugacy class of $p$ in the Galois group $\Gal(M_d/\QQ)$). We will now describe a case where we can prove that the map $\phi_d$ is indeed Frobenian and compute the density of the fibers of $\phi_d$.

Suppose that
\[
\rk_{2}\CL(d) = \rk_2\CL^+(d),
\]
which occurs if and only if $d$ has no prime factors congruent to $3$ modulo $4$, as well as if and only if the \textit{genus field} of $\K_{d}$, i.e., the subfield of $H^+_d$ that is invariant under $2\CL^+(d)$, is totally real. Further suppose that $p$ is congruent to $1$ modulo $4$ and that
\[
\rk_4\CL^+(d) = \rk_4\CL^+(dp) = 0.
\]
We then have
\[
h(d) = h^+(d) = 2^{t-1}, \quad h(p) = h^+(p) = 1, \quad\text{and}\quad h(dp) = h^+(dp) = 2^t,
\]
so that the formula \eqref{eq:Kuroda} becomes
\begin{equation}\label{eq:simplified}
h(d, p) = Q(\K_{d, p})\cdot 2^{2t-3}.
\end{equation}
We will first prove that $h^+(d, p) = 2^{2t-2}$, so that $Q(\K_{d, p}) = 2$ or $Q(\K_{d, p}) = 1$ depending on whether or not $H^+_{d, p}$ is totally real. 

\begin{theorem}\label{thm1}
With notation as above, assume that $\rk_2\CL(d)=\rk_2\CL^+(d)$ and $\rk_4\CL^+(d)=\rk_4\CL^+(dp)=0$. Then $\CL^+(d, p) \cong C_2^{2m_{d, p}}\times C_4^{t-m_{d, p}-1}$. In particular, $h^+(d, p) = 2^{2t-2}$, $Q(\K_{d, p})\in \{1, 2\}$, and $Q(\K_{d, p}) = 2$ if and only if $H^+_{d, p}$ is totally real.
\end{theorem} 

Furthermore, after proving Theorem~\ref{thm1}, we will explicitly construct $H^+_{d, p}$ as the compositum of $t-1$ disjoint quadratic extensions of the totally real field $H_{dp}$, so that $H^+_{d, p}$ is totally real if and only if each of the $t-1$ aforementioned quadratic extensions is totally real. Roughly speaking, we can prove that $m_{d, p}$ of those extensions are totally real with probability $1/2$, and we expect the remaining $t-m_{d, p}-1$ to behave similarly. Hence we make the following conjecture. 
\begin{conjecture}\label{conj1}
With notation as above, assume that $\rk_2\CL(d)=\rk_2\CL^+(d)$ and $\rk_4\CL^+(d)=0$. Let $m\in\{0, 1, \ldots, t-1\}$, let $\sP_{d, m}$ denote the set of prime numbers $p$ that are coprime to $d$, congruent to $1$ modulo $4$, satisfy $\rk_4\CL^+(dp)=0$, and satisfy $m_{d, p} = m$. Then
\[
\lim_{X\rightarrow\infty}\frac{|\{p\in\sP_{d, m}:\ p\leq X,\ Q(\K_{d, p}) = 2\}|}{|\{p\in\sP_{d, m}:\ p\leq X\}|} = \frac{1}{2^{t-1}}.
\] 
\end{conjecture}
\noindent Our main ``statistical'' result about $Q(\K_{d, p})$ is the following theorem.
\begin{theorem}\label{thm2}
With notation as in Conjecture~\ref{conj1}, the map 
$$
\sP_{d, m}\rightarrow \{1, 2\},\qquad p\mapsto Q(\K_{d, p})
$$
is Frobenian for $m = t-1$ and $m = t-2$. Moreover, Conjecture~\ref{conj1} holds for $m = t-1$ and $m = t-2$, and, for all $m\in\{0, 1, \ldots, t-3\}$, we have
\[
\limsup_{X\rightarrow\infty}\frac{|\{p\in\sP_{d, m}:\ p\leq X,\ Q(\K_{d, p}) = 2\}|}{|\{p\in\sP_{d, m}:\ p\leq X\}|} \leq \frac{1}{2^{m}}.
\] 
\end{theorem} 

\section{\texorpdfstring{The $2$-rank of $\CL^+(d, p)$}{The 2-rank of CL+(d,p)}}\label{section:2rank}
Let $d$ be a positive squarefree integer as in the introduction and Theorems~\ref{thm1} and~\ref{thm2} and let $p$ be a prime number in $\sP_{d, m}$. We begin by constructing an unramified at all \textit{finite} primes $C_2^{t+m-1}$-extension of $\K_{p, d}$ and stating a criterion for this extension to be totally positive. The methods to do so are classical by now. So as not to be overly repetitive, we will try to quote the work of Fouvry and Kl\"{u}ners \cite{FouvryKluners} whenever possible. Later, in the course of constructing certain unramified at all finite primes $C_4$-extensions of $\K_{d, p}$ and deriving a criterion for them to be totally real, we will do genuinely new work to generalize the aforementioned methods.  

Let $q_1,\ldots, q_t$ be the prime divisors of $d$. We may reorder the $q_i$ so that $\leg{q_i}{p}=1$ for $1\leq i\leq m$ and $\leg{q_i}{p}=-1$ for $m+1\leq i\leq t$. First, genus theory for the quadratic number field $\K_d$ implies that $\K_{p, q_1, \ldots, q_t}$ is an unramified at \textit{all} primes $C_2^{t-1}$-extension of $\K_{d, p}$. Now suppose that $1\leq i\leq m$, so that $\leg{q_i}{p} = 1$. Applying \cite[Lemma 19, p.2059]{FouvryKluners} with $D_1 = q_i$ (or $4q_i$ if $q_i=2$) and $D_2 = p$, we can choose $x_i, y_i, z_i\in\ZZ$ satisfying the ternary quadratic equation
\[
x_i^2 -py_i^2- q_iz_i^2  = 0
\]
such that
(i) $x_i^2$, $py_i^2$, and $q_iz_i^2$ are pairwise coprime, $y_i, z_i\geq 0$ (ii) $x_i$ odd, and one of $y_i$ and $z_i$ is even, and (iii) $x_i-y_i\equiv 1\bmod 4$ if $y_i$ is even and $x_i-z_i\equiv 1\bmod 4$ if $z_i$ is even. We define
\begin{equation}\label{eq:defalpha}
\alpha_i = 
\begin{cases}
x_i+y_i\sqrt{p} &\text{ if } z_i\text{ is odd},\\
\frac{1}{2}(x_i+y_i\sqrt{p}) &\text{ if } z_i\text{ is even};
\end{cases}
\end{equation}
then \cite[Lemma 20, p.2060]{FouvryKluners} implies that $\K_{p, q_i}(\sqrt{\alpha_i})/\QQ$ is a $D_8$-extension, unramified at all finite primes over $\K_{pq_i}$ and a fortiori over $\K_{p, q_i}$. The extension $\K_{p, q_i}(\sqrt{\alpha_i})/\K_p$ is a $(C_2\times C_2)$-extension, and so, upon taking the compositum over all $1\leq i\leq m$ and also with $\K_{p, q_1, \ldots, q_t}$, we find that
\begin{equation}\label{eq:defE}
\sE_{d, p} = \K_{p, q_1, \ldots, q_t}(\sqrt{\alpha_1}, \ldots, \sqrt{\alpha_m})
\end{equation}
is normal over $\K_p$ with Galois group isomorphic to $C_2^{t+m}$. We hence conclude that $\sE_{d, p}/\K_{d, p}$ is a normal, unramified at all finite primes extension with Galois group isomorphic to $C_2^{t+m-1}$. Since $\K_p$ has odd class number, genus theory over $\K_p$ \cite[Lemma 2.3]{Yue2} implies that $\rk_2\CL^+(d, p)$ is equal to one less than the number of primes of $\K_p$ that ramify in $\K_{d, p}$, and this is $t+m-1$. Hence we have proved
\begin{lemma}\label{lemma:genus}
Define $\sE_{d, p}$ as in \eqref{eq:defE}. Then $\sE_{d, p}/\K_{d, p}$ is the maximal unramified at all finite primes abelian extension of $K$ of exponent $2$. In particular, $\rk_2\CL^+(d, p) = t+m-1$.
\end{lemma}
\noindent Now \cite[Proposition 5, p.2061]{FouvryKluners} implies that $\K_{p, q_i}(\sqrt{\alpha_i})$ is totally real if and only if
\begin{equation}\label{eq:pqi4}
\fpr{p}{q_i}\fpr{q_i}{p} = 1.
\end{equation}
Here, as in \cite[p.2061]{FouvryKluners}, for a prime $\ell$ and a rational integer $a$, we define
\[
\fpr{a}{\ell} =
\begin{cases}
\hfil 1 & \text{if }a\text{ is a fourth power modulo }\ell \\
-1 & \text{otherwise}
\end{cases}
\]
whenever $\ell$ is an odd prime such that $\leg{a}{\ell} = 1$ and
\[
\fpr{a}{2} =
\begin{cases}
1 & \text{if }a\equiv 1\bmod 16 \\
-1 & \text{if }a\equiv 9\bmod 16
\end{cases}
\]
whenever $a\equiv 1\bmod 8$. Thus $\sE_{d, p}$ is totally real if and only if \eqref{eq:pqi4} holds for all $i\in\{1, \ldots, m\}$. We will now rewrite the condition \eqref{eq:pqi4} in terms of genuine fourth power residue symbols $\left(\frac{\cdot}{\cdot}\right)_4$ over $\K_{-1}$, a field containing a primitive fourth root of unity. Suppose that $p$ and $q_i$ split into primary primes as $p = \pi\overline{\pi}$ and $q_i = \rho_i\overline{\rho_i}$ in the ring of Gaussian integers $\OO_{\K_{-1}}$ (assume for the moment that $q_i\neq 2$). Then, since $\pi\OO_{\K_{-1}}$ and $\rho_i\OO_{\K_{-1}}$ are primes of degree $1$, we have
\[
\fpr{p}{q_i}\fpr{q_i}{p}= \leg{p}{\rho_i}_{4}\leg{q_i}{\pi}_{4} =  \leg{\pi}{\rho_i}_{4}\leg{\overline{\pi}}{\rho_i}_{4} \leg{\rho_i}{\pi}_{4}\leg{\overline{\rho_i}}{\pi}_{4}.
\]
Quartic reciprocity law \cite[Theorem 2, p 123]{IrelandRosen} implies that
\[
\leg{\pi}{\rho_i}_{4} = \leg{\rho_i}{\pi}_{4}\cdot (-1)^{\frac{p-1}{4}\frac{q_i-1}{4}}\quad\text{and}\quad\leg{\overline{\pi}}{\rho_i}_{4} = \leg{\rho_i}{\overline{\pi}}_{4}\cdot (-1)^{\frac{p-1}{4}\frac{q_i-1}{4}}.
\]
Hence
\[
\fpr{p}{q_i}\fpr{q_i}{p} 
= \leg{\rho_i}{\pi}_{4}\overline{\leg{\overline{\rho_i}}{\pi}}_{4}\leg{\rho_i}{\pi}_{4}\leg{\overline{\rho_i}}{\pi}_{4} 
= \leg{\rho_i}{\pi}_{2}.
\]
Hence we have proved that when $2\nmid q_1\dots q_m$, $\sE_{d, p}$ is totally real if and only if $p$ splits completely in the number field
\begin{equation}\label{eq:M2}
\sM_{2; d} 
=
\K_{-1, q_1, \ldots, q_m}(\sqrt{\rho_1}, \ldots, \sqrt{\rho_m}).
\end{equation}

Now suppose $q_1=2$, so that $p\equiv 1\bmod 8$. By definition, $\fpr{p}{2}=1$ if and only if $p\equiv 1\bmod 16$, i.e., if and only if $p$ splits completely in $\K_{-1, -2}(\sqrt{2+\sqrt{2}})$, while $\fpr{2}{p}=1$ if and only if $p$ splits completely in $\K_{-1, 2}(\sqrt[4]{2})$. Hence $\fpr{p}{2}\fpr{2}{p} = 1$ if and only if $p$ splits completely in $\K_{-1, 2}(\sqrt[4]{2}\sqrt{2+\sqrt{2}}) = \K_{-1, 2}(\sqrt{1+\sqrt{-1}})$. Thus, if $q_1\cdots q_m$ is even with $q_1 = 2$, say, then again $\sE_{d, p}$ is totally real if and only if $p$ splits completely in $\sM_{2; d}$, where now $\rho_1 = 1+\sqrt{-1}\in\K_{-1}$.

Define $\sP_{d, m}$ as in Conjecture~\ref{conj1} and suppose the prime $p$ as above is an element of $\sP_{d, m}$. It follows from R\'{e}dei's classical work on the $4$-rank of class groups of quadratic fields \cite{RedeiMatrix} that the condition $\rk_4\CL^+(dp) = 0$ can be detected by the Frobenius conjugacy class of $p$ in the abelian Galois group $\Gal(\K_{q_1, \ldots, q_t}/\QQ)$; furthermore, since $p$ splits completely in $\K_{q_1, \ldots, q_m}/\QQ$, the condition $\rk_4\CL^+(dp) = 0$ is in fact equivalent to $\Frob_{\K_{q_{m+1}, \ldots, q_t}/\QQ}(p)$ belonging to some fixed subset $\Sigma\subset \Gal(\K_{q_{m+1}, \ldots, q_{t}}/\QQ)$. For each element $\sigma\in\Sigma$, let $\sP_{d, m, \sigma}$ be the set of $p$ in $\sP_{d, m}$ such that $\Frob_{\K_{q_{m+1}, \ldots, q_t}/\QQ}(p) = \sigma$. Since $p\in\sP_{d, m}$ splits completely in $\K_{-1, q_1, \ldots, q_m}/\QQ$, since $\K_{q_{m+1}, \ldots, q_t}$ is disjoint from $\sM_{2; d}$, and since $[\sM_{2; d} :\K_{-1, q_1, \ldots, q_m}] = 2^m$, the Chebotarev Density Theorem implies that, for each $\sigma\in\Sigma$, the natural density of primes $p$ in $\sP_{d, m, \sigma}$ such that $\sE_{d, p}$ is totally real is equal to $2^{-m}$. Taking the union over all $\sigma\in\Sigma$, we deduce also that the natural density of primes $p$ in $\sP_{d, m}$ such that $\sE_{d, p}$ is totally real is equal to $2^{-m}$. In conjunction with Theorem~\ref{thm1}, since $H^+_{d, p}$ cannot be totally real unless $\sE_{d, p}$ is totally real, this proves the case $m_0 = t-1$ (with $\sM_{2; d}$ as the governing field) as well as the upper bound in the second part of Theorem~\ref{thm2}.

\section{\texorpdfstring{The $4$-rank of $\CL^+(d, p)$}{The 4-rank of CL+(d,p)}}
Define $\alpha_i$ as in~\eqref{eq:defalpha}. Let $\widetilde{\alpha_i}$ be the conjugate of $\alpha_i$ in $\K_p$. Let $\qq_i$ be a prime above $q_i$ in $\K_p$ and $\widetilde{\qq_i}$ be its conjugate if $i\leq m$, so that $\alpha_i\OO_{\K_p}$ factorizes into $\qq_i$ times a square ideal.

Call $a\in \K_p^\times/(\K_p^\times)^2$ a \emph{decomposition of second type for} $\K_{d, p}$ if 
\begin{itemize}
\item $a\equiv\prod_{i=1}^m \alpha_i^{e_i}\prod_{i=1}^m \widetilde{\alpha_i}^{e_i'}\prod_{i=m+1}^t q_i^{f_i}\mod (\K_p^\times)^2$, where $e_i,e_i',f_i\in\{0,1\}$; and
\item $(a,d/a)_{\rrr}=1$ for all finite and infinite primes $\rrr$ in $\OO_{\K_p}$.
\end{itemize}

\begin{lemma}
Let $a\in \K_p^\times/(\K_p^\times)^2$. Suppose that $L/\K_{d, p}$ is a $C_4$-extension unramified at all finite primes and containing $\K_{d, p}(\sqrt{a})\subseteq \sE_{d, p}$.
Then 
\begin{itemize}
\item $\Gal(L/\K_{p})\cong D_8$; and
\item $(a,d/a)_{\rrr}=1$ for all finite and infinite primes $\rrr$ in $\OO_{\K_{p}}$.
\end{itemize}
\end{lemma}

\begin{proof}
Since $\K_{p}$ has odd class number, it has no non-trivial cyclic extensions that are unramified at all finite primes.
This follows from the proofs of \cite[Lemma~15]{FouvryKluners} and \cite[Lemma~17]{FouvryKluners} with $\QQ$ replaced by $\K_{p}$.
\end{proof}


The set of decompositions of second type form a multiplicative group in $\K_p^{\times}/\left(\K_p^{\times}\right)^2$ of size $2^{1+\rk_4 \CL^+(\K_{d, p})}$.

\subsection{Generalised R\'{e}dei matrix}
Similar to \cite[Lemma 13]{FouvryKluners}, the condition $(a,d/a)_{\rrr}=1$ for all finite and infinite primes $\rrr$ in $\OO_{\K_p}$ is equivalent to the following conditions
\begin{itemize}
\item $a>0$;
\item $\Frob_{\K_{p, a}/\K_p}(\qq)=1\quad\text{if }\ord_{\qq}(da)\text{ is odd}$; and
\item $\Frob_{\K_{p, d/a}/\K_p}(\qq)=1\quad\text{if }\ord_{\qq}(a)\text{ is odd}$.
\end{itemize}

\subsubsection{Rational decompositions of second type}\label{section:Redeimatrix0}
Consider the subset of decompositions of second type where $a\in\QQ$, $a>0$.
Studying the splitting of primes in the $C_2^2$-extension $\K_{a, p}/\QQ$, we see that the condition $(a,d/a)_{\rrr}=1$ for any prime ideal $\rrr$ in $\K_p$ is equivalent to asking for each prime $q\mid d$ with $\leg{q}{p}=1$ to satisfy
\begin{align*}
\leg{a}{q}&=1\quad\text{if }  q\Bigm| \frac{d}{a} 	 ,&
\leg{d/a}{q}&=1\quad\text{if }q\mid a.
\end{align*}
Writing $a$ as a product of $q_i$, the conditions can be packaged in a matrix over $\FF_2$.
Define
\[
B_0:=
 \begin{pmatrix}
  \leg{d/q_1}{q_1} &  \leg{q_2}{q_1} & \cdots & \leg{q_m}{q_1} & \leg{q_{m+1}}{q_1} & \cdots & \leg{q_t}{q_1} \\
    \leg{q_1}{q_2} & \leg{d/q_2}{q_2} &\cdots & \leg{q_m}{q_2} & \leg{q_{m+1}}{q_2} & \cdots & \leg{q_t}{q_2} \\
\vdots  & \vdots  & \ddots &  \vdots& \vdots &  & \vdots  \\
  \leg{q_1}{q_m}  & \leg{q_2}{q_m}  & \cdots & \leg{d/q_m}{q_m} & \leg{q_{m+1}}{q_m} &\cdots  &  \leg{q_t}{q_m} 
 \end{pmatrix}_+,
\]
where the subscript $+$ denotes the conversion of each entry from $\{\pm 1\}$ to $\{0,1\}$.
Then $\ker B_0$ corresponds to the set of decompositions of the second type. The size of the matrix implies that $\dim\ker B_0\geq t-m$. Therefore $\rk_4\CL^+(\K_{d, p})\geq t-m-1$. 
Combining with the fact that $\rk_2\CL^+(\K_{d, p})=t+m-1$, the $2$-part of $\CL^+(\K_{d, p})$ has size at least $2^{2t-2}$.

\subsubsection{General decompositions of second type}\label{section:Redeimatrix}
Now consider all decompositions of second type for $\K_{d, p}$.
The set of decompositions of the second type $a$ is given by the kernel of the matrix
\[A:=
 \begin{pmatrix}
A_{11} & A_{12} & A_{13}\\
A_{21} & A_{22} & A_{23}\\
A_{31} & A_{32} & A_{33}\\
\end{pmatrix},
\]
where
\begin{footnotesize}
\begin{gather*} 
 A_{11}=
\begin{pmatrix}
\leg{d/\alpha_1}{\qq_1} & \cdots & \leg{\alpha_m}{\qq_1}\\
\vdots & \ddots & \vdots\\
\leg{\alpha_1}{\qq_m} & \cdots & \leg{d/\alpha_m}{\qq_m}\\
\end{pmatrix}_+
,\
A_{12}=
\begin{pmatrix}
\leg{\widetilde{\alpha_1}}{\qq_1} & \cdots & \leg{\widetilde{\alpha_m}}{\qq_1}\\
\vdots & & \vdots\\
\leg{\widetilde{\alpha_1}}{\qq_m} & \cdots & \leg{\widetilde{\alpha_m}}{\qq_m}\\
\end{pmatrix}_+
,\
A_{13}=
\begin{pmatrix}
\leg{q_{m+1}}{\qq_1} & \cdots & \leg{q_t}{\qq_1}\\\
\vdots & & \vdots\\
\leg{q_{m+1}}{\qq_m} & \cdots & \leg{q_t}{\qq_m}\\
\end{pmatrix}_+
,\\
A_{21}=
\begin{pmatrix}
\leg{\alpha_1}{\widetilde{\qq_1}} & \cdots & \leg{\alpha_m}{\widetilde{\qq_1}}\\
\vdots & & \vdots\\
\leg{\alpha_1}{\widetilde{\qq_m}} & \cdots & \leg{\alpha_m}{\widetilde{\qq_m}}\\
\end{pmatrix}_+
,\
A_{22}=
\begin{pmatrix}
\leg{d/\widetilde{\alpha_1}}{\widetilde{\qq_1}} & \cdots & \leg{\widetilde{\alpha_m}}{\widetilde{\qq_1}}\\
\vdots & \ddots & \vdots\\
\leg{\widetilde{\alpha_1}}{\widetilde{\qq_m}} & \cdots & \leg{d/\widetilde{\alpha_m}}{\widetilde{\qq_m}}\\
\end{pmatrix}_+
,\
A_{23}=
\begin{pmatrix}
\leg{q_{m+1}}{\widetilde{\qq_1}} & \cdots & \leg{q_t}{\widetilde{\qq_1}}\\
\vdots & & \vdots\\
\leg{q_{m+1}}{\widetilde{\qq_m}} & \cdots & \leg{q_t}{\widetilde{\qq_m}}\\
\end{pmatrix}_+
,\\
A_{31}=
\begin{pmatrix}
\leg{\alpha_1}{\qq_{m+1}} & \cdots & \leg{\alpha_m}{\qq_{m+1}}\\
\vdots & & \vdots\\
\leg{\alpha_1}{\qq_t} & \cdots & \leg{\alpha_m}{\qq_t}\\
\end{pmatrix}_+
,\ 
A_{32}=
\begin{pmatrix}
\leg{\widetilde{\alpha_1}}{\qq_{m+1}} & \cdots & \leg{\widetilde{\alpha_m}}{\qq_{m+1}}\\
\vdots & & \vdots\\
\leg{\widetilde{\alpha_1}}{\qq_{m+1}} & \cdots & \leg{\widetilde{\alpha_m}}{\qq_{m+1}}\\
\end{pmatrix}_+
,\
A_{33}=
\begin{pmatrix}
\leg{d/q_{m+1}}{\qq_{m+1}} & \cdots & \leg{q_t}{\qq_{m+1}}\\
\vdots & \ddots & \vdots\\
\leg{q_{m+1}}{\qq_t} & \cdots & \leg{d/q_t}{\qq_t}\\
\end{pmatrix}_+.
\end{gather*}
\end{footnotesize}
The matrix $A$ has the same rank as 
\[B:=
\begin{pmatrix}
A_{11} &A_{11}+A_{12} & A_{13}\\
A_{11}+A_{21} &A_{11}+A_{12}+A_{21}+A_{22} & A_{13}+A_{23}\\
A_{31} &A_{31} +A_{32} & A_{33}\\
\end{pmatrix}
=
\begin{pmatrix}
A_{11} & B_0\\
B_0^T & 0\\
\end{pmatrix}.
\]

In particular, when $B_0$ has maximal rank $m$ and $m<t$, $\rank B=2\rank B_0$, then the dimension of $\ker B=(t+m)-2m=t-m$, and so $\rk_4\CL^+(\K_{d, p})\leq t-m-1$.

\subsection{\texorpdfstring{The vanishing of the $8$-rank of $\CL^+(d, p)$}{The vanishing of the 8-rank of CL+(d,p)}}
In the following lemma, let $\CL(L)$ denote the class group of a number field $L$. 
\begin{lemma}\label{lemma:quadbound}
Let $K$ be a biquadratic number field with quadratic subfields $k_1$, $k_2$, $k_3$. Let $n\geq 1$ be an integer. 
If $\rk_{2^n}\CL^+(k_i)$ is $0$ for $i=1,2,3$,
then the $2^{n+1}$-rank of $\CL^+(K)$ is $0$.
\end{lemma}

\begin{proof}
Take a prime ideal $\PP$ in $\OO_K$ above prime $p$. It suffices to show that the order of $[\PP]^2\in\CL^+(K)$ divides the order of some ideal class in $\CL^+(k_i)$. We split into three possible cases according to the splitting of $p$ in $K$.

Suppose $p$ is inert in $k_1$, $k_2$ and splits in $k_3$. Let $\pp$ be an ideal below $\PP$ in $k_3$. Since $\PP=\pp\OO_K$, if $\pp^l$ is principal in $k_3$, then $\PP^l$ must also be principal in $K$. Therefore the order of $[\PP]\in\CL^+(K)$ divides the order of $[\pp]\in\CL^+(k_3)$.

Now suppose $p$ ramifies in $k_1$ and $k_2$. Let $\pp$ be an ideal below $\PP$ in $k_3$. Since $\PP^2=\pp\OO_K$, the order of $[\PP]^2\in\CL^+(K)$ divides the order of $[\pp]\in\CL^+(k_3)$.

Suppose instead $p$ splits completely in $K$.
Let $\pp_i$ be the prime ideal below $\PP$ in $k_i$ for $i=1,2,3$. Then $\pp_i\OO_K=\PP\PP_i$, where $\PP_i$ is a conjugate prime ideal of $\PP$ under the non-trivial map in $\Gal(K/k_i)$.
Then $(p)\OO_K=\PP\PP_1\PP_2\PP_3$, so
$[\pp_1\pp_2\pp_3\OO_K]=[\PP]^2$ in $\CL^+_K$.
Therefore the order of $[\PP]^2\in\CL^+(K)$ divides the lcm of orders of $[\pp_i]\in\CL^+(k_i)$.
\end{proof}

\begin{lemma}\label{lemma:quadbound2}
Define $\sP_{d, m}$ as in Conjecture~\ref{conj1} and suppose $p\in \sP_{d, m}$. Then $\rk_8\CL^+(d, p) = 0$.
\end{lemma}

\begin{proof}
The quadratic subfields of the biquadratic field $\K_{d, p}$ are $\K_d$, $\K_p$, and $\K_{dp}$. By assumption on $d$, we have $\rk_4\CL^+(d) = 0$. We have $\rk_2\CL^+(p) = 0$ and so also $\rk_4\CL^+(p) = 0$. By definition of $\sP_{d, p}$, we have $\rk_4\CL^+(dp) = 0$. The result now follows from Lemma~\ref{lemma:quadbound}.
\end{proof}

\subsection{Proof of Theorem~\ref{thm1}}
The lower and upper bounds from Sections~\ref{section:Redeimatrix0} and \ref{section:Redeimatrix}, respectively, yield
$$
\rk_4\CL^+(d, p) = t-m-1.
$$
In conjunction with Lemma~\ref{lemma:quadbound2}, we conclude that
$$
\CL^+(d, p)\cong C_2^{2m}\times C_4^{t-m-1}.
$$
Now equation \eqref{eq:simplified} implies that $2^{2t-2} = h^+(d, p)\geq h(d, p) = Q(\K_{d, p})\cdot 2^{2t-3}$, so that $Q(\K_{d, p})\leq 2$ with equality if and only if $h^+(d, p) = h(d, p)$, i.e., if and only if $H^+(d, p)$ is totally real.

\section{\texorpdfstring{Construction of $H^+_{d, p}$}{Construction of H+{d,p}}}
In this section, we will give an explicit construction the narrow $2$-Hilbert class field $H^+_{d, p}$ of $\K_{d, p}$. We have
$$
H^+_{d, p} = \sE_{d, p}\quad\text{ if }\quad m = t-1,
$$
where $\sE_{d, p}$ is defined in~\eqref{eq:defE}. This follows from the upper bound on the $4$-rank of $\CL^+(d, p)$ given in Section~\ref{section:Redeimatrix}.

When $m \leq t-2$, we will construct certain unramified at finite primes $C_4$-extensions of $\K_{d, p}$ by working over $\K_p$. In general, this does not lead to simple criteria for $H^+_{d, p}$ to be totally real. When $m = t-2$, then we can construct the unramified at finite primes $C_4$-extension of $\K_{d, p}$ by working over $\QQ$, and in this case we can find a criterion for $H^+_{d, p}$ to be totally real that is amenable to density computations.

\subsection{\texorpdfstring{Constructing unramified $C_4$-extensions}{Constructing unramified C4-extensions}}
We use the following proposition to show that certain extensions are unramified.
\begin{prop}[{\cite[Theorem 120]{Hecke}}]\label{prop:Hecke}
Let $L$ be a number field and $\alpha\in L\setminus L^2$ chosen coprime to $2$. Then $L(\sqrt{\beta})/L$ is unramified at all primes if and only if  $\beta\OO_L$ is a square and
\[X^2=\beta\bmod 4\]
is solvable for some $X\in L$.
\end{prop}

\begin{lemma}\label{unram}
Suppose that $a\mid d$ and that $a$ is even if $d$ is even. Let $p\equiv 1\bmod 4$ be a prime.
Suppose that
\begin{equation}\label{2decom}
X^2-aY^2=\frac{d}{a}Z^2\end{equation}
is solvable for some $X,Y,Z\in\K_p$. 
Then there exists a solution such that $\beta:=X+Y\sqrt{a}$ gives an extension $\K_{d, p, a}(\sqrt{\beta})/\K_{d, p, a}$ that is unramified at all finite primes.

More specifically, $\beta$ can be taken such that
\begin{itemize}
\item $X,Y,Z\in\OO_{\K_p}$,
\item $\gcd((X),(Y),(Z))$ is a square ideal,
\item $X, Z$ are coprime with $2$ and $2\mid Y$, 
\item $
\left\{\begin{array}{ll}X-Y &\text{if } a\equiv 1\bmod 4,\\
X-\frac{Y^2}{2} &\text{if } a\equiv 2\bmod 4
\end{array}\right\}
\equiv 
\begin{cases}
\hfil 1\bmod 4& \text{if } p\equiv 1\bmod 8,\\
1\text{ or }\frac{\frac{1+p}{2}\pm\sqrt{p}}{2}\bmod 4& \text{if } p\equiv 5\bmod 8.
\end{cases}$
\end{itemize}
\end{lemma}

\begin{proof}
Our goal is to find a suitable $\beta=X+Y\sqrt{a}$ that satisfies the requirement in Proposition~\ref{prop:Hecke}.
Let $\sigma$ be the generator of $\Gal(\K_p/\QQ)$.
Clearing denominators we can assume $X,Y,Z\in\OO_{\K_p}$.
Since the fundamental unit in $\K_p$ has norm $-1$, we can take $x,y\in\ZZ$  satisfying $x^2-py^2=-1$ and set $u=x+y\sqrt{p}$. Looking at $x^2-py^2\equiv -1\bmod 4$ we see that $x$ is even and $y$ is odd, so $u=x+y\sqrt{p}\equiv\pm\sqrt{p}\bmod 4$ in $\OO_{\K_p}$. 

{\bfseries Choosing $\beta$ to be coprime to $2$.} 
Removing factors of $2$ we can assume $2$ divides at most one of $X,Y,Z$.
If $p\equiv 5\bmod 8$, then $2$ is inert in $\K_p$ so at most one of $X,Y,Z$ is even.

If $p\equiv 1\bmod 8$, then $2$ splits in $\K_p$. Then 
\[\Norm\left(\frac{1+\sqrt{p}}{2}\right)=\frac{1-p}{4},\text{ and }
\Norm\left(\frac{3+\sqrt{p}}{2}\right)=\frac{9-p}{4}\]
are both even but defer by $2$, so one must be congruent to $2\bmod 4$. Say $\gamma$ is the element from above with norm $2\bmod 4$. Then exactly one of the primes above $2$ divides $\gamma$ with order $1$, call this prime $\pt$. 
Suppose $\max\{\ord_{\pt}X,\ord_{\pt}Y,\ord_{\pt}Z\}=k$, then take $X(\gamma^{\sigma}/2)^k,Y(\gamma^{\sigma}/2)^k,Z(\gamma^{\sigma}/2)^k$.
Repeat the same for the ideal $\pt^{\sigma}$.
Then we can assume that no prime above $2$ divides $\gcd((X), (Y), (Z))$. Therefore at least one of $X,Y,Z$ is coprime with $2$.

The squares modulo $4$ in $\OO_{\K_p}$ are $0$, $1$, $\omega:=((1+p)/2+\sqrt{p})/2$ and $\omega':=((1+p)/2-\sqrt{p})/2$.
We have $X^2=2Y^2+Z^2\bmod 4$ when $a$ is even, and $X^2=Y^2+Z^2\bmod 4$ when $a$ is odd, we see that the possible combinations are
\begin{numcases}
{(X^2,\{Y^2,Z^2\})\equiv}
(1,\{0,1\})\bmod 4, \label{mod4case1}\\
(\omega,\{0,\omega\})\bmod 4,\label{mod4case2}\\
(\omega',\{0,\omega'\})\bmod 4,\label{mod4case3}\\
(1,\{\omega,\omega'\})\bmod 4 & if $p\equiv 1\bmod8$.\label{mod4case4}
\end{numcases}

The cases \eqref{mod4case2} and \eqref{mod4case3} are only possible when $p\equiv 5\bmod 8$. For if $p\equiv 1\bmod 8$, the norm of $\omega$ and $\omega'$ are $(1-p)^2/16$, which is even, contradicting with the assumption that at least one of $X,Y,Z$ is coprime with $2$. 

For case \eqref{mod4case4}, one can obtain another solution to \eqref{2decom} that satisfies one of \eqref{mod4case1}, \eqref{mod4case2}, \eqref{mod4case3}.
Without loss of generality assume $Z^2\equiv \omega\bmod 4$. Since $X^2\equiv 1\bmod4$ implies $X\equiv \pm 1$ or $\pm\sqrt{p}\bmod 4$, multiplying $X,Y,Z$ by a suitable $\delta\in\{\pm 1, \pm u\}$, we can also assume $X\equiv 1\bmod 4$. Let $\pt$ denote the prime above $2$ such that $\pt\mid Y$, then $t^{\sigma}\mid Z$ and $\pt,\pt^{\sigma}\nmid X$.
Take 
\begin{align}
(X',Y',Z')&
:=\left(\frac{1+d/a}{2}X\pm \frac{d}{a}Z\ ,\frac{1-d/a}{2}Y,\ \frac{1+d/a}{2}Z\pm X\right)\label{transform}\\
&\ \equiv
\begin{cases}
\hfil (X\pm Z,\ 0,\ Z\pm X)\bmod 4 & \text{if }\frac{d}{a}\equiv 1\bmod 8,\\
(-X\pm Z,\ 2,\ -Z\pm X)\bmod 4& \text{if }\frac{d}{a}\equiv 5\bmod 8.\end{cases}\nonumber \end{align}
Then 
\[\Norm(X')\equiv\Norm(Z')\equiv 1\pm (Z+Z^{\sigma})+ZZ^{\sigma}\bmod 4.\]
$Z^2\equiv \omega\bmod 4$ implies $Z\equiv \pm (1+\sqrt{p})/2$ or $\pm (5+\sqrt{p})/2\bmod 4$.
Therefore $Z+Z^{\sigma}\equiv 1 \bmod 4$ and $ZZ^{\sigma}\equiv 0$ or $2 \bmod 4$. Pick the sign such that 
$\Norm(X')\equiv\Norm(Z')\equiv 2\bmod 4$. Then $\ord_{\pt}X'=\ord_{\pt}Z'=1$ and $\pt^{\sigma}\nmid X'Z'$. Carry out the reduction as before we can obtain new $X'$ and $Z'$ that are both coprime to $2$.
Therefore we can assume $X$ is always coprime with $2$ and exactly one of $Y,Z$ is coprime with $2$.

If $X,Y$ are coprime with $2$ and $2\mid Z$, the transformation \eqref{transform} give $X',Z'$ that are coprime to $2$ and $2\mid Y'$.
Therefore we can always take $X,Z$ coprime to $2$ and $2\mid Y$.  In particular $\beta\OO_{\K_{d, p, a}}$ is coprime to $2$ since its norm is odd.

{\bfseries Choosing $\beta$ to be a square ideal.} 
Let $h$ be the class number of $\K_p$, which is odd. If $\ord_{\pp}(\gcd((X),(Y),(Z))$ is odd for some prime ideal $\pp$, we can multiply $X,Y,Z$ by some $\gamma$, where $\gamma$ satisfies $\pp^h=(\gamma)$. Remove any rational prime $p$ dividing $\gcd((X),(Y),(Z))$. Therefore we can assume $\gcd((X),(Y),(Z))$ is a square ideal involving only prime ideals above odd primes that splits in $\K_p/\QQ$. For each odd prime $p$ that splits in $\K_p/\QQ$ at most one of the primes above $p$ can divide 
$\gcd((X),(Y),(Z))$ .

Suppose there exists an odd prime dividing $\beta\OO_{\K_{d, p, a}}$, then there must be a prime $\PP$ below in $\OO_{\K_{p, a}}$ dividing $\beta\OO_{\K_{p, a}}$.
Without loss of generality assume $\PP\nmid d/a$, otherwise consider the prime in $\OO_{\K_{p, d/a}}$ and interchange the roles of $a$ and $d/a$ in the following. 
Let $\pp$ is a prime in $\K_p$ below $\PP$. Taking norms to $\K_p$ we have $\pp\mid Z^2d/a$, so $\pp\mid Z$. But $\pp$ cannot divide both $X$ and $Y$ with an odd power, otherwise $\pp$ divides $\gcd((X),(Y),(Z))$ with an odd power, so $\pp\OO_{\K_{p, a}}$ cannot divide $\beta$ with an odd power. Let $\tau$ be the generator of $\Gal(\K_{p, a}/\K_p)$. Then $\ord_{\PP}\beta+ \ord_{\PP^{\tau}}\beta=\ord_{\PP}(X+Y\sqrt{a})+ \ord_{\PP}(X-Y\sqrt{a})=\ord_{\PP}Z^2=2\ord_{\pp}Z$ being even implies that $\ord_{\PP}\beta$ is even. Therefore $\beta\OO_{\K_{d, p, a}}$ has even valuation at odd primes.

{\bfseries Choosing $\beta$ to be a square modulo $4$.} 
We now handle the ramification at $2$ in cases \eqref{mod4case1}, \eqref{mod4case2} and \eqref{mod4case3}. 
First suppose $a$ is odd so $a\equiv 1\bmod 4 $.
We assumed $2\mid Y$ so \[X+Y\sqrt{a}\equiv X-Y+2Y\left(\frac{1+\sqrt{a}}{2}\right)\equiv X-Y \bmod 4.\]
Also
$(X-Y)^2=X^2+2XY+Y^2\equiv X^2\bmod 4$.

In case \eqref{mod4case1}, $X-Y\equiv \pm 1$ or $\pm \sqrt{p}\equiv \delta\bmod 4$ for some $\delta\in \{\pm 1,\pm u\}$. In case \eqref{mod4case1}, multiplying each of $X,Y,Z$ by $\delta$ satisfies the requirement since this forces 
$ X-Y\equiv 1\bmod 4$, which a square modulo $4$.

The cases \eqref{mod4case2} and \eqref{mod4case3} are only possible when $p\equiv 5\bmod 8$. 
Suppose we are in case \eqref{mod4case2}, then $(X-Y)^2\equiv\omega\bmod 4$. Then $X-Y\equiv \pm(1+\sqrt{p})/2$ or $\pm(5+\sqrt{p})/2\bmod 4$. One of $\{\pm(1+\sqrt{p})/2$ or $\pm(5+\sqrt{p})/2\}$ is a square modulo $4$, and $u(1+\sqrt{p})/2\equiv(5+\sqrt{p})/2\bmod 4$. Therefore there exist $\delta\in \{\pm 1,\pm u\}$ such that $\delta(X-Y)$ is a square modulo $4$. Replace $X,Y,Z$ by $\delta X,\delta Y,\delta Z$ then $\beta$ is a square modulo $4$. Case \eqref{mod4case3} is similar.

Now suppose $a$ is even so $a\equiv 2\bmod 4$. 
In cases \eqref{mod4case1}, \eqref{mod4case2} and \eqref{mod4case3}, similar to above there exists $\delta\in\{\pm 1,\pm u\}$ such that $\delta X\equiv X^2+Y^2/2\equiv X^2+aY^2/4$. Then $\delta X+\delta Y\sqrt{a}\equiv (X+Y\sqrt{a}/2)^2\bmod 4$.
\end{proof}

Recall from Sections~\ref{section:Redeimatrix0} and \ref{section:Redeimatrix} that the space of decompositions of the second type has dimension equal to $t-m$ inside the $\FF_2$-vector space $\K_p^\times/(\K_p^\times)^2$. Let $a_1, \ldots, a_{t-m-1}, d$ denote a basis for this space, with $a_i|d$, and let $\beta_1, \ldots, \beta_{t-m-1}$ denote the corresponding solutions constructed in Lemma~\ref{unram}. Then we can realize $H^+(d, p)$ as the field
$$
H^+(d, p) = \sE_{d, p}(\sqrt{\beta_{1}}, \ldots, \sqrt{\beta_{t-m-1}}),
$$
where $\sE_{d, p}$ is defined in \eqref{eq:defE}.
\subsection{\texorpdfstring{The case $m = t-2$}{The case m=t-2}}
Throughout this section, we take $m=t-2$, so that $q_1,\dots,q_t$ satisfies $\leg{p}{q_1}=\dots =\leg{p}{q_{t-2}}=1$ and $\leg{p}{q_{t-1}}=\leg{p}{q_t}=-1$.
\begin{lemma}\label{lemma:quadraticform}
There exists a positive integer $a\mid d$, $a\neq 1$ or $d$, 
such that 
\begin{equation}\label{eq:2decom}
px^2-ay^2=\frac{d}{a}z^2
\end{equation}
is solvable for $x,y,z\in\QQ$.
Also 
\[\leg{a}{p}=\leg{d/a}{p}=-1.\]
\end{lemma}
\begin{proof}
Since $\rank B_0=t-2$, $\dim\ker B_0=t-(t-2)=2$. We can pick $(e_1,\dots,e_t)\in\ker B_0\setminus\{\{0,\dots,0\},\{1,\dots,1\}\}$. Take $a=q_1^{e_1}\dots q_t^{e_t}$ and $b=d/a$. 
Then for each $1\leq i\leq t-2$,
\[(a,b)_{q_i}=1.\]

If $d$ is odd, $a\equiv b\equiv 1\bmod4$, so $(a,b)_2=1$.
Hilbert reciprocity implies
\begin{equation}\label{eq:q1q2}
(a,b)_{q_{t-1}}(a,b)_{q_t}=\prod_{r\leq\infty}(a,b)_r=1.
\end{equation}
When $d$ is even, $2$ is one of $q_1,\dots,q_{t-2}$ if $p\equiv 1\bmod 8$, and one of $q_{t-1},q_t$ if $p\equiv 5\bmod 8$, so \eqref{eq:q1q2} still holds.

Without loss of generality assume $q_{t-1}\mid b$, otherwise interchange $a$ and $d/a$.
Since $\rk_4\CL^+(d)=0$ and $\rk_4\CL^+(dp)=0$,
there are no decompositions of second type for $\K_d$ or $\K_{dp}$, so
\[(a,b)_{q_{t-1}}=(a,b)_{q_{t}}=-1\text{ and }\left((pa,b)_{q_{t-1}}=(pa,b)_{q_{t}}=-1\text{ or } (pa,b)_p=-1\right).\]
If $q_t\mid b$,
then $(pa,b)_p=\leg{b}{p}=\prod_{q\mid b} \leg{q}{p}=1$, so we must have $(pa,b)_{q_{t-1}}=(pa,b)_{q_{t}}=-1$,
but this contradicts with
\[(p,b)_{q_{t-1}}=(p,b)_{q_{t}}=\leg{q_{t-1}}{p}=\leg{q_t}{p}=-1.\]
Therefore $q_t\mid a$. Again $a$ cannot give a decomposition of second type for $\K_d$, so
$(a,b)_{q_{t-1}}=(a,b)_{q_t}=-1$, and hence $(pa,b)_{q_{t-1}}=(pa,b)_{q_t}=1$.
We also have 
\[\leg{a}{p}=\prod_{q_i\mid a}\leg{q_i}{p}=-1
\qquad \text{ and }\qquad 
(pa,pb)_p=\leg{ab}{p}=\prod_{q_i\mid d}\leg{q_i}{p}=1.\]
Therefore $(pa,pb)_r=1$ for any prime $r\leq \infty$.
\end{proof}

Since the $2$-part of the class group and narrow class group of $\K_p$ are both trivial, the fundamental unit in $\K_p$ has norm $-1$, we can take $u,v\in\ZZ$ satisfying 
\begin{equation}\label{eq:unit}
u^2-pv^2=-1.
\end{equation}
Looking at $u^2-pv^2\equiv -1\bmod 4$ we see that $u$ is even and $v$ is odd, so $u+v\sqrt{p}\equiv\pm\sqrt{p}\bmod 4\OO_{\K_p}$. Replacing $v$ with $-v$ if necessary we can assume $v-u\equiv 1\bmod 4$, so that $u+v\sqrt{p}\equiv\sqrt{p}\bmod 4\OO_{\K_p}$.
From $u^2-pv^2\equiv -1\bmod 8$, we see that this choice implies
\begin{equation}\label{eq:uv}
(u,v)
\equiv
\begin{cases}
(0,1)\bmod4 &\text{ if }p\equiv 1\bmod8,\\
(2,3)\bmod4 &\text{ if }p\equiv 5\bmod8.
\end{cases}
\end{equation}

If we take some $\beta=(x\sqrt{p}+y\sqrt{2})(u+v\sqrt{p})$, where $x,y$ satisfy \eqref{eq:2decom}, then $\K_{d, p, a}(\sqrt{\beta})/\K_{d, p}$ is a $C_4$-extension by the following lemma from \cite[Chapter VI, Exercise 4, p.321]{Lang}.

\begin{lemma}\label{lemma:nor}
Let $K_0$ be a number field.
Let $E = K_0(\sqrt{a})$, where $a\in K_0^{\times }\setminus (K_0^{\times })^2$, and let $F = E(\sqrt{\beta})$, where $\beta\in E^{\times }\setminus (E^{\times })^2$. Let $N = \Norm_{E/K_0}(\beta)$. Then $N\notin (K_0^{\times })^2\cup a\cdot (K_0^\times )^2$ if and only if $F/K_0$ has normal closure $F(\sqrt{N})$ and $\Gal(F(\sqrt{N})/K_0)$ is a dihedral group of order $8$.
\end{lemma}

We claim that when $\beta$ is chosen appropriately, the $\K_{d, p, a}(\sqrt{\beta})/\K_{d, p}$ is unramified at all finite primes. Note that $\K_{d, p, a}$ is contained in $H^+(d, p)$ so $\K_{d, p, a}/\K_{d, p}$ is unramified.

\begin{lemma}\label{lemma:unram}
Let $d\in\ZZ$ be a squarefree and has no prime factors congruent to $3\bmod 4$.  Suppose $a\mid d$ and $a$ is even if $d$ is even. Let $p\equiv 1\bmod 4$ be a prime.
Suppose \eqref{eq:2decom} is solvable for some $x,y,z\in\QQ$. 
There exists $x,y,z\in\ZZ$ satisfying \eqref{eq:2decom} such that
\begin{itemize}
\item $\gcd(x,y,z)=1$,
\item $x, z$ are odd and $y$ is even, and 
\item $
x-y\equiv 
1\bmod 4$.
\end{itemize}
Setting $\beta=(x\sqrt{p}+y\sqrt{a})(u+v\sqrt{p})$ gives an extension $\K_{d, p, a}(\sqrt{\beta})/\K_{d, p, a}$ that is unramified at all finite primes.
\end{lemma}

\begin{proof}
Our goal is to find a suitable $\beta=X+Y\sqrt{a}$ that satisfies the requirement in Proposition~\ref{prop:Hecke}.
Let $\sigma$ be the generator of $\Gal(\K_p/\QQ)$.
Clearing denominators we can assume $x,y,z\in\ZZ$.

{\bfseries Choosing $\beta$ to be coprime to $2$.} 
Removing factors of $2$ we can assume $2$ divides at most one of $x,y,z$.
Taking $px^2-ay^2=\frac{d}{a}z^2\bmod 4$, we see that $x$ must be odd and one of $y,z$ is even.
If $d$ is even, $a\equiv 2\bmod 8$ so $y$ must be even.
Now suppose $d$ is odd.
If $x,y$ are odd and $z$ is even, then we can take instead
\[\left(\frac{a+d/a}{2}px^2 ,\ \frac{a-d/a}{2}y+\frac{d}{a}z,\ \frac{a-d/a}{2}z-ay\right)\equiv (1,0,1)\bmod 2\]
as another set of solution to \eqref{eq:2decom}. 
Therefore we can always take $x,z$ odd and $y$ even.  In particular $\beta\OO_{\K_{d, p, a}}$ is coprime to $2$ since its norm 
\[\Norm_{K/\K_p}\beta=(u+v\sqrt{p})^2\frac{d}{a}z^2\]
is odd.

{\bfseries Choosing $\beta$ to be a square ideal.} 
We can assume $\gcd(x,y,z)=1$ by removing any common divisors.

Suppose there exists an odd prime dividing $\beta\OO_{\K_{d, p, a}}$, then there must be a prime $\PP$ below in $\OO_{\K_{p, a}}$ dividing $\beta\OO_{\K_{p, a}}$.
Without loss of generality assume $\PP\nmid d/a$, otherwise consider the prime in $\OO_{\K_{p, d/a}}$ and interchange the roles of $a$ and $d/a$ in the following. 
Let $\pp$ is a prime in $\K_p$ below $\PP$. Taking norms to $\K_p$ we have $\pp\mid z^2d/a$, so $\pp\mid z$. But $\pp$ cannot divide both $x$ and $y$, so $\pp\OO_{\K_{p, a}}$ cannot divide $\beta$. Then $\ord_{\PP}\beta=\ord_{\PP}z^2=2\ord_{\pp}z$ is even. Therefore $\beta\OO_{\K_{d, p, a}}$ has even valuation at odd primes.

{\bfseries Choosing $\beta$ to be a square modulo $4$.} 
First suppose $a$ is odd so $a\equiv 1\bmod 4 $.
We assumed $y$ is even so 
\[\begin{split}
\beta=(x\sqrt{p}+y\sqrt{a})(u+v\sqrt{p})
&\equiv (x\sqrt{p}+y\sqrt{a})\sqrt{p}
\equiv x+y\sqrt{ap}\\
&\equiv x-y+2y\left(\frac{1+\sqrt{ap}}{2}\right)\equiv x-y \bmod 4\OO_{\K_{d, p}}.
\end{split}\]
Since $x$ is odd and $y$ is even, taking $-x$ instead if necessary, we can assume $x-y\equiv 1\bmod 4$.
Then $\beta$ is a square modulo $4$ in $\OO_{\K_{d, p}}$.

Now suppose $a$ is even so $a\equiv 2\bmod 4$. Taking $-x$ instead if necessary, we can assume $x-y^2/2\equiv 1\bmod 4$. Then 
\[\beta=(x\sqrt{p}+y\sqrt{a})(u+v\sqrt{p})
\equiv x+y\sqrt{ap}
\equiv x-y\sqrt{a}+2y\sqrt{a}\left(\frac{1+\sqrt{p}}{2}\right)
\equiv \left(1-\frac{y}{2}\sqrt{a}\right)^2
\bmod 4.\]
Since $y$ is even, $y^2/2\equiv y\bmod4$.
\end{proof}

Take \eqref{eq:2decom} modulo $8$, 
if $d$ is odd, 
the choice in Lemma~\ref{lemma:unram} implies
\begin{equation}\label{eq:xy}
(x,y)\equiv
\begin{cases}
(1,0)\bmod 4&\text{ if }bp\equiv 1\bmod 8,\\
(3,2)\bmod 4&\text{ if }bp\equiv 5\bmod 8.
\end{cases}
\end{equation}

\subsection{\texorpdfstring{Criterion for $H^+_{d, p}$ to be totally real when $m = t-2$}{Criterion for H+_{d, p} to be totally real when m=t-2}}

\begin{lemma}\label{lemma:infinity}
The field $\K_{d, p, a}(\sqrt{\beta})$ is totally real if and only if $xv>0$.
\end{lemma}
\begin{proof}
Let 
$\beta_1=(x\sqrt{p}+y\sqrt{a})(u+v\sqrt{p})$, 
$\beta_2=(x\sqrt{p}-y\sqrt{a})(u+v\sqrt{p})$,
$\beta_3=(-x\sqrt{p}+y\sqrt{a})(u-v\sqrt{p})$, and
$\beta_4=(-x\sqrt{p}-y\sqrt{a})(u-v\sqrt{p})$.
Then
$\beta_1 \beta_2 = bz^2(u+v\sqrt{p})^2>0$,  
$\beta_1 \beta_3 = bz^2>0$, and
$\beta_3 \beta_4 = bz^2(u-v\sqrt{p})^2>0$,
so $\beta_1,\beta_2,\beta_3,\beta_4$ are always of the same sign.
Since
$\beta_1 +\beta_2 +\beta_3 +\beta_4= 4xvp$, we have
$\beta_1,\beta_2,\beta_3,\beta_4>0$ if and only if $xv>0$.
\end{proof}

If $d$ is odd, \eqref{eq:uv} and \eqref{eq:xy} implies 
\begin{equation}\label{eq:xv}
xv\equiv 
\begin{cases}
1\bmod 4&\text{ if }b\equiv 1\bmod 8,\\
3\bmod 4&\text{ if }b\equiv 5\bmod 8.
\end{cases}
\end{equation}

\begin{lemma}\label{lemma:keycriterion}
The field $\K_{d, p, a}(\sqrt{\beta})$ is totally real if and only if 
\[\fpr{ab}{p}\fpr{ap}{b}\fpr{bp}{a}=-1,\]
where $b=d/a$.
\end{lemma}
\begin{proof}
By Lemma~\ref{lemma:infinity}, it suffices to show that
\begin{equation}\label{eq:totallyreal}
\fpr{ab}{p}\fpr{ap}{b}\fpr{bp}{a}
=
\begin{cases}
- 1 &\text{ if } xv>0,\\
\hfil 1 &\text{ if } xv<0.
\end{cases}\end{equation}

Without loss of generality, assume $b$ is odd. Take $a=2^ja_0$, where $j=0$ if $d$ is odd and $j=1$ if $d$ is even.
Take \eqref{eq:2decom} modulo $p$ and modulo each odd $q\mid d$, we get 
\begin{gather*}
\fpr{-ab}{p}\leg{y}{p}=\leg{bz}{p}=-\leg{z}{p},\\
\fpr{ap}{b}\leg{y}{b}=\leg{px}{b}=-\leg{x}{b},\\
\fpr{bp}{a_0}\leg{z}{a_0}=\leg{px}{a_0}=-\leg{x}{a_0}.
\end{gather*}
Multiply these equations together
\begin{equation}\label{eq:delta1}
-\fpr{-ab}{p}\fpr{ap}{b}\fpr{bp}{a_0}=\leg{x}{a_0b}\leg{y}{bp}\leg{z}{a_0p}.
\end{equation}
Write $y=2^iy_0$, where $y_0$ is odd. Since $a_0\equiv b\equiv p\equiv 1\bmod 4$, we can rewrite \eqref{eq:delta1} as
\begin{equation}\label{eq:delta2}
-\fpr{-ab}{p}\fpr{ap}{b}\fpr{bp}{a}=\leg{a_0b}{|x|}\leg{2}{bp}^i\leg{bp}{y_0}\leg{a_0p}{z}.
\end{equation}
Take \eqref{eq:2decom} modulo each prime $r\mid x$, $r\mid y_0$, $r\mid z$, we get 
\[\leg{a}{|x|}=\leg{-b}{|x|}=\leg{-1}{|x|}\leg{b}{|x|}, \qquad 
\leg{p}{y_0}=\leg{b}{y_0}\qquad \text{and}\qquad 
\leg{p}{z}=\leg{a}{z}.\]
By \eqref{eq:xy}, $i=1$ and $\leg{2}{bp}=-1$ if $bp\equiv 5\bmod 8$ and
$\leg{2}{bp}=1$ if $bp\equiv 1\bmod 8$.
Simplifying \eqref{eq:delta2} gives
\begin{equation}\label{eq:delta3}
-\fpr{-ab}{p}\fpr{ap}{b}\fpr{bp}{a_0}
=\leg{-1}{|x|}\leg{2}{bp}^i\leg{2}{xz}^j
=(-1)^{\frac{|x|-1}{2}}(-1)^{\frac{p-1}{4}+\frac{b-1}{4}}\leg{2}{xz}^j.
\end{equation}
Take \eqref{eq:unit} modulo each prime $r\mid v$,
we have
$\leg{-1}{|v|}=1$,
so $|v|\equiv 1\bmod 4$.
Since 
\[
\fpr{-1}{p}=(-1)^{\frac{p-1}{4}},
\]
we have
\begin{equation}\label{eq:delta4}
\fpr{ab}{p}\fpr{ap}{b}\fpr{bp}{a_0}
=-(-1)^{\frac{|xv|-1}{2}}(-1)^{\frac{b-1}{4}}\leg{2}{xz}^j.
\end{equation}

When $d$ is odd, $j=0$, so we get \eqref{eq:totallyreal} by \eqref{eq:xv}.
When $d$ is even, $j=1$ and $a$ is even. 
Take \eqref{eq:2decom} modulo $16$ gives
$px^2\equiv bz^2\bmod 16$ if $y\equiv 0\bmod 4$, and $px^2\equiv 8+bz^2\bmod 16$ if $y\equiv 2\bmod 4$.
Then
\[\leg{2}{xz}=
(-1)^{\frac{(xz)^2-1}{8}}
=
\begin{cases}
\hfil 1&\text{ if } \hfil x^2\equiv z^2\bmod 16,\\
-1&\text{ if } x^2\equiv 9z^2\bmod 16
\end{cases}
=(-1)^{\frac{y}{2}}\fpr{pb}{2}.\]
From \eqref{eq:delta4} and since $p\equiv b\bmod 8$ here, we have
\[\fpr{ab}{p}\fpr{ap}{b}\fpr{bp}{a}
=-(-1)^{\frac{|xv|+y-1}{2}}(-1)^{\frac{p-1}{4}}.
\]
The choice $x-y\equiv 1\bmod 4$ in Lemma~\ref{lemma:unram}, together with \eqref{eq:uv}, implies
$(-1)^{\frac{xv+y-1}{2}}= (-1)^{\frac{p-1}{4}}$.
Therefore \eqref{eq:totallyreal} holds.
\end{proof}

\section{Proof of Theorem~\ref{thm2}}
Recall that we proved the cases $m= t-1$ and the upper bound in the second half of Theorem~\ref{thm2} at the end of Section~\ref{section:2rank}. It remains to prove the case $m = t-2$ of Theorem~\ref{thm2}. 
\subsection{Construction of a governing field}
We start by converting the criterion in Lemma~\ref{lemma:keycriterion} to a splitting condition in a suitable governing field.
If $p$ is a prime number congruent to $1$ modulo $4$, then we can write $p = \pi\overline{\pi}$ for some $\pi\equiv 1\bmod (1+\sqrt{-1})^3$ in $\ZZ[\sqrt{-1}]$; in this case, the inclusion $\ZZ\hookrightarrow \OO_{\QQ(\sqrt{-1})}$ induces an isomorphism $\ZZ/(p)\cong\OO_{\QQ(\sqrt{-1})}/(\pi)$, so that an integer $n$ is a fourth power modulo $p$ exactly when it is a fourth power modulo $\pi$. 
For each $1\leq i\leq t$, fix $\rho_i$ such that $q_i = \rho_i\overline{\rho_i}$ with $\rho_i\equiv 1\bmod (1+i)^3$ if $q_i\equiv 1\bmod 4$, and take $\rho_i=1+\sqrt{-1}$ if $q_i=2$.

Assuming $d=ab$ is odd for now, we have 
\begin{align*}
\fpr{d}{p}\fpr{bp}{a}\fpr{ap}{b} 
&= \fpr{d}{p}\cdot \prod_{q_i\mid a}\fpr{bp}{q_i}\cdot \prod_{q_j\mid b}\fpr{ap}{q_j} \\
&= \leg{d}{\pi}_{4}\prod_{\rho_i\mid a}\leg{bp}{\rho_i}_{4}\cdot \prod_{\rho_j\mid b}\leg{ap}{\rho_j}_{4} 
= \delta(a,b)\cdot \leg{d}{\pi}_{4}\cdot \prod_{\rho_k\mid d} \leg{p}{\rho_k}_{4},
\end{align*}
where
\begin{equation}\label{eq:defdelta}
\delta(a,b)
:= \prod_{\rho_i\mid a}\leg{b}{\rho_i}_{4}\cdot \prod_{\rho_j\mid b}\leg{a}{\rho_j}_{4}.
\end{equation}
Using quartic reciprocity as well as the fact that
\[
\leg{\rho_k}{\overline{\pi}}_{4} 
= \overline{\leg{\overline{\rho_k}}{\pi}}_{4} 
= \leg{\overline{\rho_k}}{\pi}_{4}^3,
\]
we have
\[
\leg{d}{\pi}_{4}\prod_{\rho_k\mid d} \leg{p}{\rho_k}_{4}
=\leg{d}{\pi}_{4}\prod_{\rho_k\mid d} \leg{\rho_k}{\pi}_{4}\leg{\overline{\rho_k}}{\pi}_{4}^3
= \leg{d}{\pi}_{2} \prod_{\rho_k\mid d} \leg{\overline{\rho_k}}{\pi}_{2} 
= \prod_{\rho_k\mid d}\leg{\rho_k}{\pi}_{2}.
\]

Now suppose $a$ is even, write $a=2a_0$. 
Define
\[\delta(a,b)=\delta(a_0,b)\cdot (-1)^{\frac{1-b}{8}}  \cdot \prod_{\rho_j\mid b}\leg{2}{\rho_j}_{4}.\]
Since $2=-\rho_i^2\sqrt{-1}$, we have
\begin{align*}
\fpr{d}{p}\fpr{bp}{a}\fpr{ap}{b} 
&= \leg{2}{\pi}_{4}\fpr{bp}{2}
\leg{a_0b}{\pi}_{4}
\prod_{\rho_k\mid a_0b}\leg{p}{\rho_k}_{4} \cdot
\prod_{\rho_i\mid a_0}\leg{b}{\rho_i}_{4}\cdot
\prod_{\rho_j\mid b}\leg{2a_0}{\rho_j}_{4} 
\\ 
&= \delta(a,b)\cdot (-1)^{\frac{b-1}{8}}
\leg{-\sqrt{-1}}{\pi}_{4}\fpr{bp}{2}\cdot 
\prod_{\rho_k\mid d}\leg{\rho_k}{\pi}_{2}
\end{align*}
Note that the assumption that $a$ is even and $(ap,bp)_2=1$ implies $p\equiv b\bmod 8$. 
Consider the possible classes of $\pi$ in $\ZZ[\sqrt{-1}]/8\ZZ[\sqrt{-1}]$ as in the proof of \cite[Proposition 7]{FouvryKluners}, which are 
\[\begin{cases}
1,\ 1+4i&\text{ if }p\equiv 1\bmod 16,\\
7+6i,\ 7+2i&\text{ if }p\equiv 5\bmod 16,\\
5,\ 5+4i&\text{ if }p\equiv 9\bmod 16,\\
3+6i,\ 3+2i&\text{ if }p\equiv 13\bmod 16.
\end{cases}
\]
Then
\[\leg{-\sqrt{-1}}{\pi}_{4}=(-1)^{\frac{1-p}{8}}\qquad
\text{ and }\qquad
\fpr{bp}{2}=(-1)^{\frac{p-b}{8}}\]
Therefore $(-1)^{\frac{b-1}{8}}
\leg{-\sqrt{-1}}{\pi}_{4}\fpr{bp}{2}=1$.

In either case we have
\[
\fpr{d}{p}\fpr{bp}{a}\fpr{ap}{b} 
= \delta(a,b)\cdot \prod_{\rho_k\mid d} \leg{\rho_k}{\pi}_{2}.
\]
We let
\begin{equation}\label{eq:defM4}
\sM_{4; d} 
=
\K_{-1, d}(\sqrt{\rho_1\cdots\rho_t}).
\end{equation}
Then
\[\fpr{d}{p}\fpr{bp}{a}\fpr{ap}{b}=\delta(a,b)\]
if and only if $\pi$ splits in $\sM_{4; d}/\K_{-1}$, if and only if $p$ splits completely in $\sM_{4; d}/\QQ$.

\subsection{Computation of densities}

Recall that $d = q_1\cdots q_t$ with $q_i\not\equiv 3\bmod 4$ distinct primes, that $\rk_2\CL(d) = \rk_2\CL^+(d)$, that $\rk_4\CL^+(d) = 0$, and that $\sP_{d, t-2}$ is the set of prime numbers $p$ such that 
\begin{itemize}
\item $p\equiv 1\bmod 4$,
\item $p\nmid d$,
\item $\rk_4\CL^+(dp) = 0$, and
\item there are exactly $t-2$ indices $i\in\{1, \ldots, t\}$ such that $\leg{q_i}{p} = 1$.
\end{itemize}
For each subset $\Omega\subset\{1, \ldots, t\}$ of cardinality $t-2$, let $\sP_{d, t-2, \Omega}$ denote the set of $p\in\sP_{d, t-2}$ such that $\leg{q_i}{p} = 1$ if and only if $i\in \Omega$. Hence
$$
\sP_{d, t-2} = \bigcup_{\substack{\Omega\subset \{1, \ldots, t\}\\ |\Omega| = t-2}}\sP_{d, t-2, \Omega}.
$$
If $\Omega_1$ and $\Omega_2$ are two distinct subsets of $\{1, \ldots, t\}$ of cardinality $t-2$, then there exists $i\in\Omega_1\setminus\Omega_2$, and so every prime $p\in\Omega_2$ satisfies $\leg{q_i}{p} = -1$, which means that $p\notin\Omega_1$. Hence the union above is disjoint, and so, to prove Theorem~\ref{thm2}, it suffices to prove for each $\Omega$ that the map
$$
\sP_{d, t-2, \Omega}\rightarrow \{1, 2\},\qquad p\mapsto Q(\K_{d, p})
$$
is Frobenian, with governing field $\sM_{\Omega}$, say, and that
\[
\lim_{X\rightarrow\infty}\frac{|\{p\in\sP_{d, m, \Omega}:\ p\leq X,\ Q(\K_{d, p}) = 2\}|}{|\{p\in\sP_{d, m, \Omega}:\ p\leq X\}|} = \frac{1}{2^{t-1}},
\] 
whenever $\sP_{d, t-2, \Omega}$ is non-empty. 
Then one can take the compositum $\sM = \prod_{\Omega}\sM_{\Omega}$ as a governing field for the map $\sP_{d, t-2}\rightarrow \{1, 2\}$ given by $p\mapsto Q(\K_{d, p})$.
If $\sP_{d, t-2, \Omega}$ is the empty set, then we may take $\sM_{\Omega} = \QQ$.
Otherwise, by re-numbering the indices, we may assume without loss of generality that $\Omega = \{1, \ldots, t-2\}$.

First, if $p\in\sP_{d, t-2,\Omega}$, then $\leg{q_{t-1}}{p} = \leg{q_t}{p} = -1$, so $p$ splits completely in 
$$
\sE=\K_{-1, q_1,\dots,q_{t-2},q_{t-1}q_t}.
$$
Conversely, any prime $p$ that splits completely in $\sE$ but not in
$$
\sL=\sE \K_{q_t} = \K_{-1,q_1,\dots,q_t}
$$
belongs to $\sP_{d, t-2, \Omega}$. 
Hence, letting $\sigma$ denote the element in $\Gal(L/\QQ)$ that fixes $\sqrt{-1}$ and $\sqrt{q_i}$ for $1\leq i\leq t-2$ and that sends $\sqrt{q_i}$ to $-\sqrt{q_i}$ for $i=t-1,\ t$ (i.e., $\sigma$ is the non-trivial element of $\Gal(\sL/\sE)$), we see that $p\in \sP_{d, t-2, \Omega}$ if and only if $\Frob_{\sL/\QQ}(p) = \sigma$.

Next, note that Lemma~\ref{lemma:quadraticform} yields the same decomposition $a$, $b = d/a$ for $\K_{d, p_1}$ and $\K_{d, p_2}$ for any two primes $p_1, p_2\in\sP_{d, t-2, \Omega}$. Also note that $\K_{-1, d}\subset \sE$, so, for primes $p$ that split completely in $\sE$, the final result of the previous section can be restated as 
\begin{equation}\label{eq:newM4}
\fpr{d}{p}\fpr{bp}{a}\fpr{ap}{b}=\delta(a,b)\Longleftrightarrow p\text{ splits completely in }\sE\sM_{4; d}/\QQ,
\end{equation}
where $\sM_{4; d}$ is as in \eqref{eq:defM4}.
Hence, by Lemma~\ref{lemma:keycriterion} and the result of the previous section, a prime $p$ is in $\sP_{d, t-2, \Omega}$ and $H^+_{d, p}$ is totally real if and only if 
\begin{itemize}
\item $\Frob_{\sL/\QQ}(p) = \sigma$,
\item $p$ splits completely in $\sE\sM_{2; d} /\QQ$, where $\sM_{2; d}$ is as in \eqref{eq:M2}, and
\item identifying $\Gal(\sE\sM_{4; d}/\sE)$ with the group $\{\pm 1\}$, and viewing $\Gal(\sE\sM_{4; d}/\sE)$ as a subgroup of $\Gal(\sE\sM_{4; d}/\QQ)$ in the canonical way, $\Frob_{\sE\sM_{4; d}/\QQ}(p) = -\delta(a, b)$.
\end{itemize}
Define $\sM$ to be the compositum 
$$
\sM = \sL\sM_{2; d}\sM_{4; d} = \K_{-1, q_1,\dots,q_t}(\sqrt{\rho_1}, \ldots, \sqrt{\rho_{t-2}},\sqrt{\rho_1\cdots\rho_t}),
$$
and observe that there is a unique element $\tau(a, b)\in\Gal(\sM/\sE) \subset \Gal(\sM/\QQ)$ depending on $a$ and $b$ such that for every prime $p$, the three conditions listed above are equivalent to the condition that $\Frob_{\sM/\sE}(\pp) = \tau(a, b)$, where $\pp$ is any prime of $\sE$ lying above $p$. 
Note also that $\Gal(\sM/\sE)\cong C_2^{t}$. 
Applying the Chebotarev Density Theorem to $\sM/\sE$ and $\sL/\sE$, we get
\begin{align*}
\lim_{X\rightarrow\infty}&\frac{|\{p\in\sP_{d, t-2, \Omega}:\ p\leq X,\ Q(\K_{d, p}) = 2\}|}{|\{p\in\sP_{d, t-2, \Omega}:\ p\leq X\}|} \\
& = \lim_{X\rightarrow\infty}\frac{|\{p\in\sP_{d, t-2, \Omega}:\ p\leq X,\ H^+_{d, p}\text{ is totally real}\}|}{|\{p\in\sP_{d, t-2, \Omega}:\ p\leq X\}|} \\
 & = \lim_{X\rightarrow\infty}\frac{\#\{\pp\text{ prime in }\OO_{\sE}:\ \Norm\pp\leq X,\ \Frob_{\sM/\sE}(\pp)=\tau(a, b)\}}{\#\{\pp\text{ prime in }\OO_{\sE}:\ \Norm\pp\leq X,\ \Frob_{\sL/\sE}(\pp)=\sigma\}} \\
 & = \frac{2^{-t}}{2^{-1}} = 2^{-t+1},
\end{align*}
as desired.

\begin{figure}
\centering
 \begin{tikzpicture}[node distance = 1cm, auto]
      \node (Q) {$\QQ$};
      \node (E) [above of=Q] {$\sE=\K_{-1, q_1,\dots,q_{t-2},q_{t-1}q_t}$};
      \node (O) [above of=E] {$\sL=\K_{-1, q_1,\dots,q_t}$};
      \node (M) [above of=O] {$\sM=\K_{-1, q_1,\dots,q_t}(\sqrt{\rho_1}, \ldots, \sqrt{\rho_{t-2}},\sqrt{\rho_1\cdots\rho_t})$};
      \draw[-] (Q) to node {} (E);
      \draw[-] (E) to node {} (O);
      \draw[-] (O) to node {} (M);
\end{tikzpicture}
\caption{Field diagram of the fields $\sE\subset\sL\subset\sM$.}
\end{figure}

\bibliographystyle{abbrv}
\bibliography{Chan_Milovic_References_UG}

\begin{thebibliography}{10}

\bibitem{AM}
A.~Azizi and A.~Mouhib.
\newblock Sur le rang du 2-groupe de classes de {$\mathbb Q(\sqrt m,\sqrt d)$}
  o\`u {$m=2$} ou un premier {$p\equiv 1\pmod 4$}.
\newblock {\em Trans. Amer. Math. Soc.}, 353(7):2741--2752, 2001.

\bibitem{BLS}
E.~Benjamin, F.~Lemmermeyer, and C.~Snyder.
\newblock On the unit group of some multiquadratic number fields.
\newblock {\em Pacific J. Math.}, 230(1):27--40, 2007.

\bibitem{CohnLag}
H.~Cohn and J.~C. Lagarias.
\newblock On the existence of fields governing the {$2$}-invariants of the
  classgroup of {${\bf Q}(\sqrt{dp})$} as {$p$} varies.
\newblock {\em Math. Comp.}, 41(164):711--730, 1983.

\bibitem{CohnLag2}
H.~Cohn and J.~C. Lagarias.
\newblock Is there a density for the set of primes {$p$} such that the class
  number of {${\bf Q}(\sqrt{-p})$} is divisible by {$16$}?
\newblock In {\em Topics in classical number theory, {V}ol. {I}, {II}
  ({B}udapest, 1981)}, volume~34 of {\em Colloq. Math. Soc. J\'anos Bolyai},
  pages 257--280. North-Holland, Amsterdam, 1984.

\bibitem{FouvryKluners}
E.~Fouvry and J.~Kl\"uners.
\newblock On the negative {P}ell equation.
\newblock {\em Ann. of Math. (2)}, 172(3):2035--2104, 2010.

\bibitem{Hecke}
E.~Hecke.
\newblock {\em Lectures on the theory of algebraic numbers}, volume~77 of {\em
  Graduate Texts in Mathematics}.
\newblock Springer-Verlag, New York-Berlin, 1981.
\newblock Translated from the German by George U. Brauer, Jay R. Goldman and R.
  Kotzen.

\bibitem{IrelandRosen}
K.~Ireland and M.~Rosen.
\newblock {\em A classical introduction to modern number theory}, volume~84 of
  {\em Graduate Texts in Mathematics}.
\newblock Springer-Verlag, New York, second edition, 1990.

\bibitem{Kubota53}
T.~Kubota.
\newblock \"{U}ber die {B}eziehung der {K}lassenzahlen der {U}nterk\"{o}rper
  des bizyklischen biquadratischen {Z}ahlk\"{o}rpers.
\newblock {\em Nagoya Math. J.}, 6:119--127, 1953.

\bibitem{Kubota56}
T.~Kubota.
\newblock \"{U}ber den bizyklischen biquadratischen {Z}ahlk\"{o}rper.
\newblock {\em Nagoya Math. J.}, 10:65--85, 1956.

\bibitem{Kuroda}
S.~Kuroda.
\newblock \"{U}ber die {K}lassenzahlen algebraischer {Z}ahlk\"{o}rper.
\newblock {\em Nagoya Math. J.}, 1:1--10, 1950.

\bibitem{Lang}
S.~Lang.
\newblock {\em Algebra}, volume 211 of {\em Graduate Texts in Mathematics}.
\newblock Springer-Verlag, New York, third edition, 2002.

\bibitem{YZ}
Y.~Ouyang and Z.~Zhang.
\newblock Hilbert genus fields of real biquadratic fields.
\newblock {\em Ramanujan J.}, 37(2):345--363, 2015.

\bibitem{RedeiMatrix}
L.~R\'edei.
\newblock Arithmetischer {B}eweis des {S}atzes \"uber die {A}nzahl der durch
  vier teilbaren {I}nvarianten der absoluten {K}lassengruppe im quadratischen
  {Z}ahlk\"orper.
\newblock {\em J. Reine Angew. Math.}, 171:55--60, 1934.

\bibitem{Sime1}
P.~J. Sime.
\newblock Hilbert class fields of real biquadratic fields.
\newblock {\em J. Number Theory}, 50(1):154--166, 1995.

\bibitem{Sime2}
P.~J. Sime.
\newblock On the ideal class group of real biquadratic fields.
\newblock {\em Trans. Amer. Math. Soc.}, 347(12):4855--4876, 1995.

\bibitem{Ste1}
P.~Stevenhagen.
\newblock Ray class groups and governing fields.
\newblock In {\em Th\'eorie des nombres, {A}nn\'ee 1988/89, {F}asc.\ 1}, Publ.
  Math. Fac. Sci. Besan\c con, page~93. Univ. Franche-Comt\'e, Besan\c con,
  1989.

\bibitem{Yue1}
Q.~Yue.
\newblock The generalized {R}\'{e}dei-matrix.
\newblock {\em Math. Z.}, 261(1):23--37, 2009.

\bibitem{Yue2}
Q.~Yue.
\newblock Genus fields of real biquadratic fields.
\newblock {\em Ramanujan J.}, 21(1):17--25, 2010.

\end{thebibliography}

\end{document}